\documentclass[a4paper]{amsart}
\usepackage{amsmath,amssymb}
\usepackage{amsbsy}
\usepackage{amsthm}
\usepackage{amscd}
\usepackage{ascmac}
\usepackage{bm}
\usepackage[all]{xy}
\usepackage{pb-diagram}
\usepackage[pdftex]{graphicx}
\usepackage{tikz}
\usepackage{here}
\usepackage{color}

\newtheorem{theo}{Theorem}[section]

\newtheorem{prop}[theo]{Proposition}
\newtheorem{lem}[theo]{Lemma}
\newtheorem{cor}[theo]{Corollary}

\newtheorem{rmk}[theo]{Remark}
\newtheorem*{nt}{Notation}
\newtheorem{prob}{Problem}
\newtheorem{case}{Case}

\title[characteristic polynomials of Abelian varieties]{The Characteristic Polynomials Of Abelian Varieties Of  Higher Dimension Over Finite Fields}
\author{DAIKI HAYASHIDA}\thanks{hayashid@math.kyoto-u.ac.jp\\\ \ \ \ Department of Mathematics, Faculty of Science, Kyoto University\\\ \ \ \ Kitashirakawa Oiwake-cho, Sakyo-ku, Kyoto 606-8502, Japan}

\begin{document}

 \maketitle
 
 \begin{abstract}
The characteristic polynomials of abelian varieties over the finite field $\mathbb{F}_q$ with $q=p^n$ elements have a lot of arithmetic and geometric information. They have been explicitly described for abelian varieties up to dimension 4, but little is known in higher dimension. In this paper, among other things, we obtain the following three results on the characteristic polynomial of abelian varieties. First, we prove a relation between $n$ and $e$, where $e$ is a certain multiplicity associated with a simple abelian variety of arbitrary dimension over $\mathbb{F}_q$. Second, we explicitly describe the characteristic polynomials of simple abelian varieties of arbitrary dimension $g$, when $e=g$. Finally, we explicitly describe the coefficients of characteristic polynomials of abelian varieties of dimension 5 over $\mathbb{F}_q$.
 \end{abstract}
 
 
 \section{Introduction}
 An abelian variety over a field $k$ is a complete group variety over $k$. It has various extremely good properties. In particular, abelian varieties over finite fields can be used to approach various practical issues, e.g. cryptography. Since the characteristic polynomial has a lot of information on abelian varieties, it is important to investigate the characteristic polynomials of abelian varieties in detail.

Let $\mathbb{F}_q$ be a finite field with $q=p^n$ elements and $X$ an abelian variety of dimension $g$ over $\mathbb{F}_q$. Fix a prime $l \neq p$ and let $T_l (X)$ be the $l$-adic Tate module of $X$. The $q$-th power Frobenius endomorphism $\pi_X : X \rightarrow X$ induces a homomorphism as $\mathbb{Z}_l$-modules
$$
T_l (\pi_X) : T_l (X) \longrightarrow T_l (X).
$$
The characteristic polynomial $f_X (t)$ of $\pi_X$ is defined by
$$
f_X (t) = {\rm det}(t-T_l (\pi_X)) ,
$$
which is known to have coefficients in $\mathbb{Z}$ independent of $l$. In this paper, we call $f_X (t)$ the characteristic polynomial of $X$ or the characteristic polynomial for short instead of the characteristic polynomial of the Frobenius endomorphism of $X$. The characteristic polynomial $f_X (t)$ is of the form
$$
f_X (t) = t^{2g} + a_1 t^{2g-1} + \cdots + a_{g-1} t^{g-1} + a_g t^g + a_{g-1} q t^{g-1} + \cdots + a_1 q^{g-1} t + q^g,
$$
where $a_1 , \cdots , a_g \in \mathbb{Z}$. The set of roots in $\mathbb{C}$ of $f_X (t)$ is in the form $\{ w_1 , \overline{w_1} , \cdots , w_g , \overline{w_g} \}$, where $w_i$ is a Weil number for $i=1,\cdots , g$. A ($q$-){\it Weil number $w$} is an algebraic integer such that for any embedding $\sigma : \mathbb{Q} (w) \hookrightarrow \mathbb{C}$, $|\sigma(w)|=\sqrt{q}$. A monic polynomial with integer coefficients whose roots are ($q$-)Weil numbers is called a ($q$-)${\it Weil\ polynomial}$. Thus the characteristic polynomial of the Frobenius endomorphism is a Weil polynomial, but the converse is not necessarily true.

Let $X,Y$ be abelian varieties defined over $\mathbb{F}_q$ and $f_X (t) , f_Y (t)$ the characteristic polynomials of $X$ and $Y$ respectively. Then, by Tate's theorem, $X$ is ($\mathbb{F}_q$-)isogenous to $Y$ if and only if $f_X (t) = f_Y (t)$. It is known that any abelian variety $X$ over $\mathbb{F}_q$ is isogenous to
$$
X_1^{r_1} \times \cdots \times X_m^{r_m},
$$
where $X_i$ is a simple abelian variety over $\mathbb{F}_q$, $X_i$ is not isogenous to $X_j$ for $i \neq j$ and $r_i \geq 1$ is an integer. If $f_{X_i}(t)$ is the characteristic polynomial of $X_i$, then 
$$
f_X (t) = f_{X_1} (t)^{r_1} \cdots f_{X_m} (t)^{r_m}.
$$
Therefore to determine characteristic polynomials of abelian varieties of dimension $g$ over finite fields, it is sufficient to determine characteristic polynomials of simple abelian varieties of dimension $\leq g$. Moreover, if $X$ is ($\mathbb{F}_q$-)simple, then $f_X (t) = m_X (t) ^e$, where $m_X (t)$ is an irreducible polynomial and $e \geq 1$ is an integer, which we call the multiplicity of $X$. From this equality, it is obvious that $e$ divides $2{\rm dim}(X)$. 

In this paper, we show three main results on characteristic polynomials of abelian varieties. First, we give a new condition on $e$ as follows.

\begin{theo}\label{theo1.1}
With the above notation, $e$ divides $n$ except for the case where $f_X (t)$ has a real root.
\end{theo}
Many studies so far of characteristic polynomials of simple abelian varieties over $\mathbb{F}_q$ are done with $n$ fixed, where the possibility of $e$ is examined only by using the divisibility condition $e \mid 2{\rm dim}(X)$ unrelated to the value of $n$. On the other hand, Theorem \ref{theo1.1} gives rise to another direction of study depending on the value of $n$. For example, assume $q=p$, then $n=1$, hence $e=1$ by Theorem \ref{theo1.1}. Thus we conclude that the characteristic polynomial of a simple abelian variety over $\mathbb{F}_p$ must be irreducible, unless it has a real root.

The second main result is the following theorem generalizing the known results \cite[Propositions 2, 3]{Xi} for dimensions 3 and 4 to arbitrary dimension. 

\begin{theo}\label{theo1.2}
Let $a, b \in \mathbb{Z}$ and $2<g \in \mathbb{Z}$. Set $f(t) = (t^2 + at + b)^g \in \mathbb{Z}[t]$. Then the polynomial $f(t)$ is the characteristic polynomial of a simple abelian variety of dimension $g$ over $\mathbb{F}_q$ with $q=p^n$ elements if and only if $g$ divides $n$, $b=q$, $|a|< 2\sqrt{q}$ and $a= kq^{s/g}$, where $k,s$ are integers satisfying $(k,p)=1$, $(s,g)=1$ and $1\leq s < g/2$.
\end{theo}

This theorem was known for dimension $g$ less than or equal to $4$ (cf. Remark \ref{rmk3.4}).

The third main result is the following theorem, which explicitly describes the characteristic polynomials of simple abelian varieties of dimension 5 over $\mathbb{F}_q$. (See the end of section 2 for the known results in dimension $\leq 4$.) Let $v_p$ denote the $p$-adic additive valuation normalized as $v_p (p) =1$.

\begin{theo}\label{theo1.3}
Let $f(t)= t^{10}+a_1 t^9 + a_2 t^8 + a_3 t^7 + a_4 t^6 + a_5 t^5 + a_4 qt^4 + a_3 q^2 t^3 + a_2 q^3 t^2 + a_1 q^4 t + q^5$ be a Weil polynomial. Then $f(t)$ is the characteristic polynomial of a simple abelian variety of dimension 5 over $\mathbb{F}_q$ if and only if one of the following conditions holds:
\begin{itemize}
\item[(I)] The polynomial $f(t)$ is of the form $f(t)=(t^2 + at + q)^5$, where $a\in \mathbb{Z}$, $|a| < 2\sqrt{q}$ and the following condition holds:
      \begin{itemize}
           \item[(1)] 5 divides $n$ and $a= kq^{s/5}$, where $k$ is an integer, $(k,p)=1$ and $s$ is $1$ or $2$.
      \end{itemize}
\item[(I\hspace{-.1em}I)] The polynomial $f(t)$ is irreducible and one of the following conditions holds:
\begin{itemize}
\item[(2)] $v_p (a_1)=0$, $v_p (a_2) \geq n/2$, $v_p (a_3) \geq n$, $v_p (a_4) \geq 3n/2$, $v_p (a_5) \geq 2n$ and $f(t)$ has no root of valuation $n/2$ nor a factor of degree $3$ in $\mathbb{Q}_p$,
\item[(3)] $v_p (a_1) =0$, $v_p (a_2) \geq n/3$, $v_p (a_3) \geq 2n/3$, $v_p (a_4) =n$, $v_p (a_5) \geq 3n/2$ and $f(t)$ has no root of valuation $n/3,n/2$ or $2n/3$ in $\mathbb{Q}_p$, 
\item[(4)] $v_p (a_1) =0$, $v_p (a_2) \geq n/4$, $v_p (a_3) \geq n/2$, $v_p (a_4) \geq 3n/4$, $v_p (a_5) = n$ and $f(t)$ has no root of valuation $n/4$ or $3n/4$ nor an irreducible factor of degree $2$ in $\mathbb{Q}_p$,
\item[(5)] $v_p (a_2) =0$, $v_p (a_3) \geq n/2$, $v_p (a_4) \geq n$, $v_p (a_5) \geq 3n/2$ and $f(t)$ has no root of valuation $n/2$ nor an irreducible factor of degree $3$ in $\mathbb{Q}_p$,
\item[(6)] $v_p (a_2) = 0$, $v_p (a_3) \geq n/3$, $v_p (a_4) \geq 2n/3$, $v_p (a_5) =n$ and $f(t)$ has no root of valuation $n/3$ or $2n/3$ in $\mathbb{Q}_p$, 
\item[(7)] $v_p (a_3) = 0$, $v_p (a_4) \geq n/2$, $v_p (a_5) \geq n$ and $f(t)$ has no root of valuation $n/2$ in $\mathbb{Q}_p$, 
\item[(8)] $v_p (a_1) \geq n/3$, $v_p (a_2 ) \geq 2n/3$, $v_p (a_3)=n$, $v_p (a_4) \geq 3n/2$, $v_p (a_5) \geq 2n$ and $f(t)$ has no root of valuation $n/3$, $n/2$ or $2n/3$ in $\mathbb{Q}_p$, 
\item[(9)] $v_p (a_4) = 0$, $v_p (a_5) \geq n/2$ and $f(t)$ has no root of valuation $n/2$ in $\mathbb{Q}_p$,
\item[(10)] $v_p (a_1) \geq n/4$, $v_p (a_2 ) \geq n/2$, $v_p (a_3)\geq 3n/4$, $v_p (a_4)=n$, $v_p (a_5) \geq 3n/2$ and $f(t)$ has no root of valuation $n/4, n/2$ or $3n/4$ and has two irreducible factors of degree $4$ in $\mathbb{Q}_p$,
\item[(11)] $v_p (a_5) = 0$,
\item[(12)] $v_p (a_1) \geq n/5$, $v_p (a_2 ) \geq 2n/5$, $v_p (a_3)\geq 3n/5$, $v_p (a_4)\geq 4n/5$, $v_p (a_5) =n$ and $f(t)$ has two irreducible factors of degree $5$ in $\mathbb{Q}_p$,
\item[(13)] $v_p (a_1) \geq 2n/5$, $v_p (a_2 ) \geq 4n/5$, $v_p (a_3)\geq 6n/5$, $v_p (a_4)\geq 8n/5$, $v_p (a_5) =2n$ and $f(t)$ has two irreducible factors of degree $5$ in $\mathbb{Q}_p$,
\item[(14)] $v_p (a_1)\geq n/2$, $v_p (a_2) \geq n$, $v_p (a_3) \geq 3n/2$, $v_p (a_4) \geq 2n$, $v_p (a_5) \geq 5n/2$ and $f(t)$ has no root of valuation $n/2$ nor a factor of degree $3$ or $5$ in $\mathbb{Q}_p$.

\end{itemize}
\end{itemize}

\end{theo}


\section{Preliminaries and problem settings}
Now we briefly review the Honda-Tate theory (cf. \cite{WM,Ki}) which is a powerful classification theory of (simple) abelian varieties over finite fields. 

We denote the set of $q$-Weil numbers by $W(q)$ and define the following equivalence relation on $W(q)$: We say that  $\pi, \pi^{\prime} \in W(q)$ are {\it conjugate} (and write $\pi \sim \pi^{\prime}$) if $\pi$ and $\pi^{\prime}$ have the same minimal polynomial over $\mathbb{Q}$.

\begin{theo}[cf. [1, Theorem 9\rm{]}]  \label{thm1.1}
There is a bijection $X \longmapsto \pi_X$ from the set of $\mathbb{F}_q$-isogeny classes of simple abelian varieties over $\mathbb{F}_q$ to the set of conjugacy classes of $W(q)$.
\end{theo}
In other words, this theorem claims that there is a one-to-one correspondence between two seemingly unrelated objects --- abelian varieties and algebraic integers.

Let $X$ be a simple abelian variety over $\mathbb{F}_q$ and $\pi_X$ the $q$-th power Frobenius endomorphism of $X$. Let $E:={\rm End}_{\mathbb{F}_q}^{0} (X) = {\rm End}_{\mathbb{F}_q} (X) \otimes_{\mathbb{Z}} \mathbb{Q}$ and $F:=\mathbb{Q} [\pi_X] \subseteq E$. Then $E$ is a division algebra whose center is $F$. The dimension of $X$ in Theorem \ref{thm1.1} satisfies the following property.
\begin{prop}[cf. [1, Theorem 8\rm{]}]\label{prop1.2}
With notation as above, the dimension of the abelian variety $X$ satisfies
$$
2{\rm dim}(X) = \sqrt{[E:F]}\cdot [F:\mathbb{Q}].
$$
\end{prop}

Let $v$ be a place of $F$, $F_v$ the completion of $F$ at $v$ and ${\rm Br}(F_v)$ the Brauer group of $F_v$. Then we have the following formula for the invariant ${\rm inv}_v (E) := {\rm inv}_{F_v} (E \otimes_F F_v ) \in {\rm Br}(F_v) \subset \mathbb{Q}/\mathbb{Z}$.

\begin{prop}[cf. [1, Theorem 8\rm{]}]\label{prop1.3}
We have
\begin{itemize}
\item ${\rm inv}_v (E) = 1/2$ if $v$ is a real place,
\item ${\rm inv}_v (E) = 0$ if $v$ is a finite place not dividing $p$ or $v$ is a complex place,
\item ${\rm inv}_v (E) = \dfrac{v_p (\pi_X)}{v_p (q)}\cdot [F_v : \mathbb{Q}_p] \ {\rm mod} \ \mathbb{Z}$ if $v$ is a place dividing $p$.
\end{itemize}
\end{prop}

The following useful lemma shows that an abelian variety corresponding to a real Weil number must be of dimension less than or equal to $2$.
\begin{lem}\label{lem2.1}
Let $X$ be a simple abelian variety over $\mathbb{F}_q$ with $q=p^n$ elements and $f_X (t)$ the characteristic polynomial. Suppose that $f_X (t)$ has a real root. Then we have
\begin{itemize}
\item if $n$ is even, then ${\rm dim}(X)=1$,
\item if $n$ is odd, then ${\rm dim}(X)=2$.
\end{itemize}
\end{lem} 
\begin{proof} (cf. \cite[5.1]{Ki})
Let $\varpi$ be a real root of $f_X (t)$. Then $F=\mathbb{Q}(\pi_X)$ is identified with $\mathbb{Q}(\varpi)$, and since $F$ has a real embedding $\sigma : F=\mathbb{Q}(\varpi) \hookrightarrow \mathbb{R} \subset \mathbb{C}$ and $\varpi$ is a Weil number, $\sigma(\varpi^2) = \sigma(\varpi) \cdot \overline{\sigma(\varpi)} = q = p^n$. So $\varpi^2 = p^n$.

First assume $n$ is even. Then $\varpi = \pm p^{n/2} \in \mathbb{Q}$, hence $F= \mathbb{Q}$. Let $E={\rm End}_{\mathbb{F}_q}^{0} (X)$. From Proposition \ref{prop1.3},  ${\rm inv}_v (E) = 1/2$ for the unique real place $v$ and 
$$
{\rm inv}_v (E) = \dfrac{v_p (\pi_X)}{v_p (q)}\cdot [F_v : \mathbb{Q}_p] =\dfrac{v_p (\varpi)}{v_p (q)}\cdot [F_v : \mathbb{Q}_p] = 1/2
$$
for the unique place dividing $p$.
Thus the order of $[E]$ in ${\rm Br}(F)$ has 2, so $\sqrt{[E:F]}=2$ (cf. \cite[Theorem 3.6]{Ki}). From Proposition \ref{prop1.2}, we obtain
$$
{\rm dim}(X) = (1/2)\sqrt{[E:F]}\cdot [F:\mathbb{Q}]=1.
$$
Next assume $n$ is odd. Then $F=\mathbb{Q}(\varpi) = \mathbb{Q}(\sqrt{p})$, and there are two real places. From Proposition \ref{prop1.3}, the invariant of $E$ at these two places is $1/2$. Since there exists only one place $v$ dividing $p$, this implies that the invariant of $E$ at $v$ dividing $p$ must be $0 \in \mathbb{Q}/\mathbb{Z}$. So $\sqrt{[E:F]}=2$. From Proposition \ref{prop1.2}, we also obtain
$$
{\rm dim}(X) = (1/2)\sqrt{[E:F]}\cdot [F:\mathbb{Q}]=2.
$$

\end{proof}

\begin{rmk}
\upshape
(i) Let $f(t)$ be an irreducible Weil polynomial and suppose that $f(t)$ has a real root $\varpi$. Then the proof of Lemma \ref{lem2.1} shows $\varpi = \pm \sqrt{q}$, hence $f(t)= t\pm \sqrt{q}$ (resp. $t^2 -q$) for $n$ even (resp. odd). Thus the degree of $f(t)$ is at most 2.

(ii) It is known that an abelian variety corresponding to a real Weil number is {\it supersingular}. See \cite[5.1]{Ki} for details.
\end{rmk}

The following lemma plays an important role in a specific situation.
\begin{lem}\cite[Proposition 2.5]{Maisner}\label{lem2.3}
Let $\mathbb{F}_q$ be a finite field with $q=p^n$ elements. Let $X$ be a simple abelian variety over $\mathbb{F}_q$ with characteristic polynomial $f_X (t) = (t^2 + at + q)^{{\rm dim}(X)}$, where $a \in \mathbb{Z}$ such that $|a|< 2\sqrt{q}$. Let $m = v_p (a)$ and $d=a^2 -4q$. Then
\begin{equation*}
{\rm dim}(X)=
\begin{cases}
\frac{n}{(m,n)} & \text{if $m<\frac{n}{2}$} \\
2 & \text{if $m \geq \frac{n}{2}$ and $d \in \mathbb{Q}_p ^{\times 2}$} \\
1 & \text{if $m \geq \frac{n}{2}$ and $d \not\in \mathbb{Q}_p ^{\times 2}$}
\end{cases}
\end{equation*}

\end{lem}
Let $X$ be a simple abelian variety of dimension $g$ over $\mathbb{F}_q$ and $f_X (t)$ the characteristic polynomial. First, note that $f_X(t)$ is a Weil polynomial of degree $2g$. Thus we must consider the following problems to determine explicitly the characteristic polynomials.

\begin{prob}\label{prob1}
Find a necessary and sufficient condition for a polynomial $f(t) = t^{2g} + a_1 t^{2g-1} + \cdots + a_{g-1} t^{g+1} + a_g t^g + a_{g-1} q t^{g-1} + \cdots + a_1 q^{g-1} t + q^g$ with integer coefficients to be a Weil polynomial.
\end{prob}
In general, this problem is very difficult and has only been solved up to dimension 5 at present.

Second, we recall that since $X$ is simple, $f_X (t) = m_X (t) ^e$, where $m_X (t)$ is an irreducible polynomial and $e\geq1$ is an integer (the multiplicity of $X$).
\begin{prob}\label{prob2}
With notation as above, determine all possible multiplicities $e$.
\end{prob} 
Lemma \ref{lem2.1} plays an important role in dealing with Problem \ref{prob2}. As another approach to Problem \ref{prob2}, Theorem \ref{theo1.1} is also important. Finally, for each $e$ examined in Problem \ref{prob2}, we consider the following problem.
\begin{prob}\label{prob3}
Find a necessary and sufficient condition for a Weil polynomial $f(t) = t^{2g} + a_1 t^{2g-1} + \cdots + a_{g-1} t^{g+1}+ a_g t^g + a_{g-1} q t^{g-1} + \cdots + a_1 q^{g-1} t + q^g$ to be the characteristic polynomial of a simple abelian variety of dimension $g$ over $\mathbb{F}_q$. Namely, for a Weil polynomial $f(t)$, find a condition on the coefficients of $f(t)$ under which there exists a simple abelian variety over $\mathbb{F}_q$ whose characteristic polynomial coincides with $f(t)$. 
\end{prob}
Here we briefly review the known results concerning the above problems. Problems 1,2 and 3 are solved in \cite[theorem 4.1]{Wa} for dimension 1; in \cite{Maisner} and \cite{R} for dimension 2; in \cite{Ha} and \cite[Proposition 1,2]{Xi} for dimension 3; and in \cite{HaSi} and \cite[Theorem 2]{Xi} for dimension 4. Problem 1 is solved in \cite{Sohn} for dimension 5.



\section{Main results}
Let $X$ be a simple abelian variety of dimension $g$ over $\mathbb{F}_q$ with $q=p^n$ elements, $\pi_X$ the $q$-th power Frobenius endomorphism of $X$ and $f_X (t)$ the characteristic polynomial of $X$. We have $f_X (t) = m_X (t) ^e$, where $m_X (t)$ is an irreducible polynomial and $e\geq1$ is an integer. Let $E:={\rm End}^{0}_{\mathbb{F}_q} (X)$ and $F:= \mathbb{Q}(\pi_X) \subset E$.

\begin{lem}\label{lem4.1}
The least common denominator of ${\rm inv}_v (E)$ for all places of $F$ is equal to $e$.
\end{lem}
\begin{proof}
The polynomial $m_X (t)$ is the minimal polynomial of $\pi_X$ over $\mathbb{Q}$ and the degree is $2g/e$. Hence $[F:\mathbb{Q}]=2g/e$. From Proposition \ref{prop1.2}, we have $\sqrt{[E:F]}=e$. Since $\sqrt{[E:F]}$ coincides with the order of $[E]$ in ${\rm Br}(F)$ (cf. \cite[Theorem 3.6]{Ki}), which is equal to the least common denominator of all invariants of $E$, we obtain the desired conclusion.
\end{proof}

\begin{cor}\label{cor3.2}
An irreducible Weil polynomial $f(t)$ of degree $2g$ is the characteristic polynomial of a simple abelian variety of dimension $g$ over $\mathbb{F}_q$ (i.e. $e=1$) if and only if $f(t)$ has no real root and the following condition (**) holds.
\begin{gather*}
(**)
\begin{cases}
\dfrac{v_p (f_i (0))}{n} \in \mathbb{Z}, \\{\rm where}\  f_i (t) \ {\rm runs}\ {\rm through}\  {\rm all}\ {\rm monic}\  {\rm irreducible}\  {\rm factors}\  {\rm of}\  f(t)\  {\rm in}\  \mathbb{Q}_p [t].
\end{cases}
\end{gather*}
\end{cor}

\begin{proof}
Since $F \cong \mathbb{Q}[t]/(f(t))$, we have $F\otimes_{\mathbb{Q}} \mathbb{Q}_p \cong \mathbb{Q}_p [t] /(f(t)) \cong \prod_{i} \mathbb{Q}_p [t]/(f_i (t))$. Thus, the set of places of $F$ dividing $p$ is in one-to-one correspondence with $\{ f_i (t) \}_i$. First, assume that an irreducible Weil polynomial $f(t)$ of degree $2g$ is the characteristic polynomial of a simple abelian variety $X$ of dimension $g$ over $\mathbb{F}_q$. Then the least common denominator of ${\rm inv}_v (E)$ is $1$ by Lemma \ref{lem4.1}. Thus $f(t)$ has no real root from Proposition \ref{prop1.3} (otherwise, $F$ admits a real place), and we have $v_p (f_i (0)) /n \in \mathbb{Z}$ from Proposition \ref{prop1.3}, where $f_i (t)$ runs through all irreducible factors of $f(t)$ in $\mathbb{Q}_p [t]$. Indeed, if $f_i (t) = (t-\alpha_1)\cdots (t-\alpha_{{\rm deg} (f_i (t))})$, then we have $f_i (0) = \pm \alpha_1\cdots \alpha_{{\rm deg} (f_i (t))}$. Since $f_i (t)$ is irreducible over $\mathbb{Q}_p$, we obtain $v_p (\alpha_1) = \cdots = v_p (\alpha_{{\rm deg} f_i (t)}) (=v_p (\pi_X))$. Hence $v_p (f_i (0)) = v_p (\pi_X) {\rm deg} (f_i (t)) = v_p (\pi_X)[F_v : \mathbb{Q}_p]$. Conversely, assume $v_p (f_i (0))/n \in \mathbb{Z}$ for all irreducible factors $f_i (t)$ in $\mathbb{Q}_p [t]$ of an irreducible Weil polynomial $f(t)$ of degree $2g$. Let $X^{\prime}$ be a simple abelian variety of dimension $g^{\prime}$ corresponding to a root of $f(t)$ in Theorem \ref{thm1.1} and $E^{\prime}={\rm End}^0 (X^{\prime})$. Then $f_{X^{\prime}} (t) = f(t) ^e$, so $2g^{\prime} =2ge$. Since $f_{X^{\prime}} (t)$ has no real root and we have
$$
{\rm inv}_v (E^{\prime}) = \dfrac{v_p (f_i (0))}{n} \in \mathbb{Z}
$$
for the place $v$ corresponding to $f_i (t)$, the least common denominator of ${\rm inv}_v (E^{\prime})$ is $1$. Thus we obtain $e=1$ by Lemma \ref{lem4.1}.
\end{proof}

In the known approach to Problem \ref{prob2}, the possibility of $e$ is examined by determining whether $f_X (t)$ has a real root or not and using the fact that $e$ divides $2g$. Here, as another approach, we propose to use a relation between $n$ and $e$ given by Theorem \ref{theo1.1}.
\begin{proof}[Proof of Theorem \ref{theo1.1}]
As $m_X (t)$ is the minimal polynomial of $\pi_X$ over $\mathbb{Q}$, $F \cong \mathbb{Q}[t]/(m_X (t))$. Considering the decomposition of $m_X (t)$ in $\mathbb{Q}_p [t]$, we have 
$$
F\otimes \mathbb{Q}_p = \mathbb{Q}_p [t] /(m_X (t)) = \prod^r_{i=1} \mathbb{Q}_p [t] /(m_i (t)),
$$
where $m_i (t) \in \mathbb{Q}_p [t]\ (i=1, \cdots, r)$ is an irreducible monic polynomial such that $m_X (t) = m_1 (t) \cdots m_r (t)$.

For each $i$, $\mathbb{Q}_p [t] /(m_i (t))= F_v$, where $F_v$ is the completion at the place $v$ above $p$ corresponding to the embedding $F \hookrightarrow \mathbb{Q}_p [t]/(m_i (t))$. Hence $[F_v : \mathbb{Q}_p]= {\rm deg}(m_i (t))$. Let $v_p$ denote the $p$-adic additive valuation normalized as $v_p (p) =1$. Since $m_i (t)$ is an irreducible polynomial over $\mathbb{Q}_p$, $v_p (\alpha)$ has the same value for all roots $\alpha$ of $m_i (t)$. Thus
$$
v_p (m_i (0)) = {\rm deg}(m_i (t)) \cdot v_p (\alpha)
$$
and from Proposition \ref{prop1.3}, we have
$$
{\rm inv}_v (E) = \dfrac{v_p (\pi_X)}{v_p (q)}\cdot [F_v : \mathbb{Q}_p] = \dfrac{v_p (m_i (0))}{n}
$$
for a place $v$ dividing $p$. Since there is no real place by assumption, the invariants of $E$ at the other places are equal to $0$ by Proposition \ref{prop1.3}.

Since $v_p (m_i (0))\in \mathbb{Z}$, $e$ divides $n$ from Lemma \ref{lem4.1}.
\end{proof}

\begin{rmk}\label{rmk3.3}
\upshape

When $f_X (t)$ has a real root, the assertion of Theorem \ref{theo1.1} does not hold in general. Indeed, assume that $n$ is odd and let $X$ be a simple abelian variety over $\mathbb{F}_q$ corresponding to the Weil number $\sqrt{q}$. As $t^2- q$ is irreducible over $\mathbb{Q}$, we may write $f_X (t) = (t^2 - q)^e$. Now, by Lemma \ref{lem2.1}, we have ${\rm dim}(X) =2$, hence $e=2$. Thus, $e$ does not divide $n$.

\end{rmk}

In general, it is difficult to describe the characteristic polynomials for all $e$ dividing $2g$. On the other hand, in the case of $e=g$, i.e. in the case that $m_X (t)$ is an irreducible polynomial of degree 2, we have Theorem \ref{theo1.2}, which describes explicitly the characteristic polynomials for arbitrary $g$. (This was only known for $g \leq 4$. cf. Remark \ref{rmk3.4} below.) 

\begin{nt}
We write ${\rm lcd}(a_1 , \cdots , a_m)$ for the least common denominator of $a_1 , \cdots , \\ a_m \in \mathbb{Q}/\mathbb{Z}$ and ${\rm d}(a)$ for the denominator of $a\in \mathbb{Q}/\mathbb{Z}$.
\end{nt}

\begin{proof}[Proof of Theorem \ref{theo1.2}]
Assume first that $g$ divides $n$, $b=q$, $|a|< 2\sqrt{q}$ and $a= kq^{s/g}$, where $k,s$ are integers satisfying $(k,p)=1$, $(s,g)=1$ and $1\leq s < g/2$. Since $f(t)$ is a Weil polynomial, there exists a simple abelian variety $X$ corresponding to a root of $f(t)$ in Theorem \ref{thm1.1}. Then we have $f_X (t) = (t^2 + at+q)^{{\rm dim}(X)}$. Since 
\begin{align*}
m:=v_p (a)&=v_p (kq^{s/g}) \\
                  &= v_p (p^{ns/g}) \ \ \ \ \ {\rm since} \ (k,p)=1, \\
                  &=ns/g < n/2, 
\end{align*}
we obtain ${\rm dim}(X) = g$ from Lemma \ref{lem2.3}.

Conversely, we assume that the polynomial $f (t) = (t^2 + at + b)^g $ is the characteristic polynomial of a simple abelian variety $X$ of dimension $g$ over $\mathbb{F}_q$. Since $f(t)$ is a Weil polynomial, we get $|a|\leq 2\sqrt{q}$ and $|b|=q$. 

First, suppose $b=-q$, then $a=0$. This implies that $f(t)$ has a real root, which contradicts $g>2$ from Lemma \ref{lem2.1}. 

Second, suppose $b=q$ and $|a|= 2\sqrt{q}$. Then $f(t)$ has a real root again. Similarly, this contradicts $g>2$. Hence we obtain $b=q$ and $|a| < 2\sqrt{q}$. Moreover, $g$ divides $n$ by Theorem \ref{theo1.1}. From Lemma \ref{lem4.1}, 
\begin{align}
\tag{*}{\it the}\ {\it least}\ {\it common}\ {\it denominator}\ {\it of}\ {\it all}\ {\it invariants}\ {\it of}\ E={\rm End}^0 (X)\  {\it is}\  g.  
\end{align}
We consider the Newton polygon for $t^2+at+q$. This has 3 possible vertices $(0,n), (1,v_p (a))$ and $(2,0)$.

Suppose the Newton polygon is a line, i.e. $v_p (a) \geq n/2$ or $a=0$. Then we can decompose $t^2 +at+q$ as $(t-\alpha_1 )(t- \alpha_2)$ in $\overline{\mathbb{Q}}_p [t]$ so that $v_p (\alpha_1 ) =v_p (\alpha_2 ) =n/2$. We have
\begin{align*}
&{\rm lcd}\left( \frac{v_p (\alpha_1)}{n}[\mathbb{Q}_p (\alpha_1): \mathbb{Q}_p] , \frac{v_p (\alpha_2)}{n}[\mathbb{Q}_p (\alpha_2): \mathbb{Q}_p]  \right) \\
&= {\rm lcd}\left( \frac{1}{2}[\mathbb{Q}_p (\alpha_1): \mathbb{Q}_p] , \frac{1}{2}[\mathbb{Q}_p (\alpha_2): \mathbb{Q}_p]  \right),
\end{align*}
which is 1 or 2 by using Proposition \ref{prop1.3}. This contradicts the fact (*) since $g>2$. 

Hence the point $(1,v_p (a))$ must be a vertex of the Newton polygon. In other words, we have $v_p (a) < n/2$ and $a \neq 0$. Then we can decompose $t^2 +at+q$ as $(t-\alpha_1 )(t- \alpha_2)$ so that $t-\alpha_1, t-\alpha_2 \in \mathbb{Q}_p [t]$, $v_p (\alpha_1 ) =n-v_p (a)$ and $v_p (\alpha_2 ) = v_p (a)$. We have
\begin{align*}
&{\rm lcd}\left( \frac{v_p (\alpha_1)}{n}[\mathbb{Q}_p (\alpha_1): \mathbb{Q}_p] , \frac{v_p (\alpha_2)}{n}[\mathbb{Q}_p (\alpha_2): \mathbb{Q}_p]  \right) \\
&= {\rm lcd}\left( 1-\frac{v_p (a)}{n}, \frac{v_p (a)}{n} \right) \\
&= {\rm d}\left(\frac{v_p (a)}{n} \right) = g
\end{align*}
if and only if $v_p (a) = ns/g$ with an integer $s$ satisfying $(s,g)=1$. Further, since $v_p (a) < n/2$, we have $s < g/2$. This implies that ($g$ divides $n$ and) $a= kq^{s/g}$, where $k,s$ are integers satisfying $(k,p)=1$, $(s,g)=1$ and $1\leq s < g/2$.
\end{proof}

\begin{rmk}\label{rmk3.4}
\upshape
Theorem 1.2 for the case of $g=1,2,3$ and $4$ is in \cite[Theorem 4.1]{Wa}, \cite[Theorem 2.9]{Maisner}, \cite[Proposition 2]{Xi} and \cite[Proposition 3]{Xi} respectively. (Note that the condition (ii) in \cite[Proposition 2]{Xi} is unnecessary.)

\end{rmk}
The following proposition gives an answer to Problem \ref{prob2} in the case of odd prime dimension.
\begin{prop}\label{prop3.5}
Let $l$ be an odd prime. Consider the characteristic polynomial of a simple abelian variety of {\it dimension $l$} over $\mathbb{F}_q$. Then $e=1$ or $e=l$.
\end{prop}
\begin{proof}
Since $e$ divides $2l$ and $l$ is a prime, we have either $e=1,2,l$ or $2l$. Suppose $e=2$ or $e=2l$. Then $m_X (t)$ is an irreducible polynomial of odd degree. Hence the polynomial $m_X (t)$ has a real root. This contradicts $l \geq 3$ by Lemma \ref{lem2.1}.
\end{proof}
In general, it is very difficult to explicitly describe the characteristic polynomial of an abelian variety of high dimension in the case of $e=1$. On the other hand, we know the characteristic polynomial in the case of $e=l$ by Theorem \ref{theo1.2}. Hence we ``almost" complete describing the characteristic polynomial of a simple abelian variety of prime dimension over $\mathbb{F}_q$ by Proposition 3.5.

Finally, we give the main result concerning the characteristic polynomial of a simple abelian variety of dimension 5.

Problem \ref{prob1} in dimension 5 is solved by G.Y.Sohn \cite[Theorem 2.1]{Sohn}. Problem \ref{prob2} in dimension 5 has been done in Proposition \ref{prop3.5}. Problem \ref{prob3} in dimension 5 corresponds to Theorem \ref{theo1.3} and we give the proof as follows.
\begin{proof}[Proof of Theorem \ref{theo1.3}]
Assume first $f(t)$ is not irreducible, i.e. $e=5$ by Proposition \ref{prop3.5}. This case has been done in Theorem \ref{theo1.2}. Next, an irreducible Weil polynomial $f(t)$ of degree 10 is the characteristic polynomial of a simple abelian variety $X$ of dimension 5 over $\mathbb{F}_q$ if and only if the condition (**) in Corollary \ref{cor3.2} holds. (Note that $f(t)$ has no real root from Remark 2.5 (i).)

Let $\mathcal{NP} (f)$ denote the Newton polygon of $f(t)$. Then $\mathcal{NP}(f)$ has 10 possible vertices $(0,5n), (1,4n+v_p (a_1)), (2, 3n+v_p (a_2)), (3, 2n+v_p (a_3)), (4, n+v_p (a_4)), (5, v_p (a_5)),$
$(6, v_p (a_4)), (7, v_p (a_3)), (8, v_p (a_2)), (9, v_p (a_1))$ and $(10,0)$. Note that if some of these points is a vertex, then the point must be a lattice point belonging to $\mathbb{Z} \times n\mathbb{Z}$. (cf. \cite[p.64]{HaSi}.) By symmetry of $\mathcal{NP}(f)$, it is sufficient to classify cases according to whether either $(1,4n+v_p (a_1)), (2, 3n+v_p (a_2)), (3, 2n+v_p (a_3)), (4, n+v_p (a_4))$ or $(5, v_p (a_5))$  is a vertex or not.
\begin{case}
{\bf \boldmath $(1,4n+v_p (a_1))$ is a sole vertex:}
\upshape

In this case, $\mathcal{NP}(f)$ is as in Figure 1. This occurs if and only if $v_p (a_1) =0, v_p (a_2 ) \geq n/2 , v_p (a_3) \geq n, v_p (a_4) \geq 3n/2$ and $v_p (a_5) \geq 2n$.  Then we can decompose $f(t)$ as $ \prod_{i=1}^{10} (t - \alpha_i)$ in $\overline{\mathbb{Q}}_p [t]$ so that
\begin{gather*} 
t- \alpha_1 ,\  t- \alpha_{10} ,\ \prod_{i=2}^{9} (t - \alpha_i)  \in \mathbb{Q}_p [t], \\
v_p (\alpha_1) = n, v_p (\alpha_{10}) = 0 , v_p (\alpha_i) = n/2\  \text{for}\  i=2,\cdots,9.
\end{gather*}
The condition (**) holds if and only if $f(t)$ has no root of valuation $n/2$ nor a factor of degree 3 in $\mathbb{Q}_p$.

\begin{figure}
  \begin{center}
    \begin{tabular}{c}   
    
         \begin{tikzpicture}
   \draw [thick, -stealth](-0.5,0)--(10.5,0) node [anchor=north]{};
   \draw [thick, -stealth](0,-0.5)--(0,5.5) node{};
   \node [anchor=north west] at (0,0) {O};

   \draw [very thick, domain=0:1, samples=200] plot(\x, -\x+5);
   \draw [very thick, domain=1:9, samples=200] plot(\x, -\x/2+9/2);
   \draw [very thick, domain=9:10, samples=200] plot(\x, 0);

   \draw [dashed](1,0) node [anchor=north]{$1$}--(1,1)--(1,5);
   \draw [dashed](2,0) node [anchor=north]{$2$}--(2,1)--(2,5);
   \draw [dashed](3,0) node [anchor=north]{$3$}--(3,1)--(3,5);
   \draw [dashed](4,0) node [anchor=north]{$4$}--(4,1)--(4,5);
   \draw [dashed](5,0) node [anchor=north]{$5$}--(5,1)--(5,5);
   \draw [dashed](6,0) node [anchor=north]{$6$}--(6,1)--(6,5);
   \draw [dashed](7,0) node [anchor=north]{$7$}--(7,1)--(7,5);
   \draw [dashed](8,0) node [anchor=north]{$8$}--(8,1)--(8,5);
   \draw [dashed](9,0) node [anchor=north]{$9$}--(9,1)--(9,5);
   \draw [dashed](10,0) node [anchor=north]{$10$}--(10,1)--(10,5);
   \draw [dashed](0,1) node [anchor=east]{$n$}--(1,1)--(10,1);
   \draw [dashed](0,2) node [anchor=east]{$2n$}--(1,2)--(10,2);
   \draw [dashed](0,3) node [anchor=east]{$3n$}--(1,3)--(10,3);
   \draw [dashed](0,4) node [anchor=east]{$4n$}--(1,4)--(10,4);
   \draw [dashed](0,5) node [anchor=east]{$5n$}--(1,5)--(10,5);
   \draw [dashed](0,5) node [anchor=east]{}--(0,5)--(10,0);
\end{tikzpicture} 
          \end{tabular}
    \caption{$(1,4n+v_p (a_1))$ is a vertex.}
  \end{center}
\end{figure} 

\end{case}
\begin{case}
{\bf \boldmath $(1,4n+v_p (a_1))$ and \boldmath $(4,n+v_p (a_4))$ are vertices:}
\upshape

In this case, $\mathcal{NP}(f)$ is as in Figure 2. This occurs if and only if $v_p (a_1) =0, v_p (a_2) \geq n/3, v_p (a_3) \geq 2n/3, v_p (a_4) =n$ and $v_p (a_5) \geq 3n/2$. Then we can decompose $f(t)$ as $\displaystyle \prod_{i=1}^{10} (t - \alpha_i)$ in $\overline{\mathbb{Q}}_p [t]$ so that
\begin{gather*} 
t- \alpha_1, \ t - \alpha_{10}  \in \mathbb{Q}_p [t], \\
 (t- \alpha_2)(t- \alpha_3)(t- \alpha_4), \ (t- \alpha_7)(t- \alpha_8)(t- \alpha_9) \in \mathbb{Q}_p [t], \\
 (t- \alpha_5)(t- \alpha_6) \in \mathbb{Q}_p [t], \\
v_p (\alpha_1) = n, v_p (\alpha_{10}) = 0 ,\\
 v_p (\alpha_2) =v_p (\alpha_3)=v_p (\alpha_4)= 2n/3, \\
 v_p (\alpha_7)=v_p (\alpha_8)=v_p (\alpha_9)= n/3, \\
 v_p (\alpha_5) = v_p (\alpha_6) = n/2.
\end{gather*}
It is obvious that $v_p (\alpha_1) /n , v_p (\alpha_{10} ) / n \in \mathbb{Z}$. The condition (**) holds if and only if $f(t)$ has no root of valuation $n/3,n/2$ or $2n/3$ in $\mathbb{Q}_p$.

\begin{figure}
  \begin{center}
    \begin{tabular}{c}   
         \begin{tikzpicture}
   \draw [thick, -stealth](-0.5,0)--(10.5,0) node [anchor=north]{};
   \draw [thick, -stealth](0,-0.5)--(0,5.5) node{};
   \node [anchor=north west] at (0,0) {O};

   \draw [very thick, domain=0:1, samples=200] plot(\x, -\x+5);
   \draw [very thick, domain=1:4, samples=200] plot(\x, -2/3*\x+14/3);
   \draw [very thick, domain=4:6, samples=200] plot(\x, -\x/2 +4);
   \draw [very thick, domain=6:9, samples=200] plot(\x, -\x/3 +3);
   \draw [very thick, domain=9:10, samples=200] plot(\x, 0);

   \draw [dashed](1,0) node [anchor=north]{$1$}--(1,1)--(1,5);
   \draw [dashed](2,0) node [anchor=north]{$2$}--(2,1)--(2,5);
   \draw [dashed](3,0) node [anchor=north]{$3$}--(3,1)--(3,5);
   \draw [dashed](4,0) node [anchor=north]{$4$}--(4,1)--(4,5);
   \draw [dashed](5,0) node [anchor=north]{$5$}--(5,1)--(5,5);
   \draw [dashed](6,0) node [anchor=north]{$6$}--(6,1)--(6,5);
   \draw [dashed](7,0) node [anchor=north]{$7$}--(7,1)--(7,5);
   \draw [dashed](8,0) node [anchor=north]{$8$}--(8,1)--(8,5);
   \draw [dashed](9,0) node [anchor=north]{$9$}--(9,1)--(9,5);
   \draw [dashed](10,0) node [anchor=north]{$10$}--(10,1)--(10,5);
   \draw [dashed](0,1) node [anchor=east]{$n$}--(1,1)--(10,1);
   \draw [dashed](0,2) node [anchor=east]{$2n$}--(1,2)--(10,2);
   \draw [dashed](0,3) node [anchor=east]{$3n$}--(1,3)--(10,3);
   \draw [dashed](0,4) node [anchor=east]{$4n$}--(1,4)--(10,4);
   \draw [dashed](0,5) node [anchor=east]{$5n$}--(1,5)--(10,5);
   \draw [dashed](0,5) node [anchor=east]{}--(0,5)--(10,0);
\end{tikzpicture}
          \end{tabular}
    \caption{$(1,4n+v_p (a_1))$ and $(4,n+v_p (a_4))$ are vertices.}
  \end{center}
\end{figure}

\end{case}

\begin{case}
{\bf \boldmath $(1,4n+v_p (a_1))$ and \boldmath $(5,v_p (a_5))$ are vertices:}
\upshape

In this case, $\mathcal{NP}(f)$ is as in Figure 3. This occurs if and only if $v_p (a_1) =0, v_p (a_2) \geq n/4, v_p (a_3) \geq n/2, v_p (a_4) \geq 3n/4$ and $v_p (a_5) = n$. Then we can decompose $f(t)$ as $\displaystyle \prod_{i=1}^{10} (t - \alpha_i)$ in $\overline{\mathbb{Q}}_p [t]$ so that
\begin{gather*} 
t- \alpha_1, \ t - \alpha_{10}  \in \mathbb{Q}_p [t], \\
 (t- \alpha_2)(t- \alpha_3)(t- \alpha_4)(t- \alpha_5) ,\ (t- \alpha_6)(t- \alpha_7)(t- \alpha_8)(t- \alpha_9) \in \mathbb{Q}_p [t] \\
v_p (\alpha_1) = n, v_p (\alpha_{10}) = 0 ,\\
 v_p (\alpha_2) =v_p (\alpha_3)=v_p (\alpha_4)=v_p (\alpha_5)= 3n/4, \\
 v_p (\alpha_6) =v_p (\alpha_7)=v_p (\alpha_8)=v_p (\alpha_9)= n/4.
\end{gather*}
It is obvious that $v_p (\alpha_1) /n , v_p (\alpha_{10} ) / n \in \mathbb{Z}$. The condition (**) holds if and only if $f(t)$ has no root of valuation $n/4$ or $3n/4$ nor an irreducible factor of degree 2 in $\mathbb{Q}_p$.

\begin{figure}
  \begin{center}
    \begin{tabular}{c}   
         \begin{tikzpicture}
   \draw [thick, -stealth](-0.5,0)--(10.5,0) node [anchor=north]{};
   \draw [thick, -stealth](0,-0.5)--(0,5.5) node{};
   \node [anchor=north west] at (0,0) {O};

   \draw [very thick, domain=0:1, samples=200] plot(\x, -\x+5);
   \draw [very thick, domain=1:5, samples=200] plot(\x, -3/4*\x+19/4);
   \draw [very thick, domain=5:9, samples=200] plot(\x, -\x/4 +9/4);
   \draw [very thick, domain=9:10, samples=200] plot(\x, 0);

   \draw [dashed](1,0) node [anchor=north]{$1$}--(1,1)--(1,5);
   \draw [dashed](2,0) node [anchor=north]{$2$}--(2,1)--(2,5);
   \draw [dashed](3,0) node [anchor=north]{$3$}--(3,1)--(3,5);
   \draw [dashed](4,0) node [anchor=north]{$4$}--(4,1)--(4,5);
   \draw [dashed](5,0) node [anchor=north]{$5$}--(5,1)--(5,5);
   \draw [dashed](6,0) node [anchor=north]{$6$}--(6,1)--(6,5);
   \draw [dashed](7,0) node [anchor=north]{$7$}--(7,1)--(7,5);
   \draw [dashed](8,0) node [anchor=north]{$8$}--(8,1)--(8,5);
   \draw [dashed](9,0) node [anchor=north]{$9$}--(9,1)--(9,5);
   \draw [dashed](10,0) node [anchor=north]{$10$}--(10,1)--(10,5);
   \draw [dashed](0,1) node [anchor=east]{$n$}--(1,1)--(10,1);
   \draw [dashed](0,2) node [anchor=east]{$2n$}--(1,2)--(10,2);
   \draw [dashed](0,3) node [anchor=east]{$3n$}--(1,3)--(10,3);
   \draw [dashed](0,4) node [anchor=east]{$4n$}--(1,4)--(10,4);
   \draw [dashed](0,5) node [anchor=east]{$5n$}--(1,5)--(10,5);
   \draw [dashed](0,5) node [anchor=east]{}--(0,5)--(10,0);
\end{tikzpicture}
          \end{tabular}
    \caption{$(1,4n+v_p (a_1))$ and $(5,v_p (a_5))$ are vertices.}
  \end{center}
\end{figure}

\end{case}
\begin{case}
{\bf \boldmath $(2,3n+v_p (a_2))$ is a sole vertex:}
\upshape

In this case, $\mathcal{NP}(f)$ is as in Figure 4. This occurs if and only if $v_p (a_1) \geq 0, v_p (a_2) =0, v_p (a_3) \geq n/2, v_p (a_4) \geq n$ and $v_p (a_5) \geq 3n/2$. Note that $v_p (a_1) \geq 0$ is a trivial condition. Then we can decompose $f(t)$ as $\displaystyle \prod_{i=1}^{10} (t - \alpha_i)$ in $\overline{\mathbb{Q}}_p [t]$ so that
\begin{gather*} 
(t- \alpha_1)(t- \alpha_2), (t- \alpha_9)(t- \alpha_{10})\in \mathbb{Q}_p [t], \\
(t - \alpha_3)(t - \alpha_4)(t - \alpha_5)(t - \alpha_6)(t - \alpha_7)(t - \alpha_8)  \in \mathbb{Q}_p [t], \\
v_p (\alpha_1) =v_p (\alpha_2)= n, \ v_p (\alpha_9)=v_p (\alpha_{10}) = 0 ,\\
 v_p (\alpha_3) =v_p (\alpha_4)=v_p (\alpha_5)=v_p (\alpha_6)= v_p (\alpha_7)= v_p (\alpha_8) =  n/2.
\end{gather*}
Since $v_p (\alpha_1) /n, v_p (\alpha_2) /n, v_p (\alpha_9) /n, v_p (\alpha_{10}) /n$ are integers, the condition (**) holds if and only if $f(t)$ has no root of valuation $n/2$ nor an irreducible factor of degree 3 in $\mathbb{Q}_p$.

\begin{figure}
  \begin{center}
    \begin{tabular}{c}   
         \begin{tikzpicture}
   \draw [thick, -stealth](-0.5,0)--(10.5,0) node [anchor=north]{};
   \draw [thick, -stealth](0,-0.5)--(0,5.5) node{};
   \node [anchor=north west] at (0,0) {O};

   \draw [very thick, domain=0:2, samples=200] plot(\x, -\x+5);
   \draw [very thick, domain=2:8, samples=200] plot(\x, -\x/2+4);
   \draw [very thick, domain=8:10, samples=200] plot(\x, 0);

   \draw [dashed](1,0) node [anchor=north]{$1$}--(1,1)--(1,5);
   \draw [dashed](2,0) node [anchor=north]{$2$}--(2,1)--(2,5);
   \draw [dashed](3,0) node [anchor=north]{$3$}--(3,1)--(3,5);
   \draw [dashed](4,0) node [anchor=north]{$4$}--(4,1)--(4,5);
   \draw [dashed](5,0) node [anchor=north]{$5$}--(5,1)--(5,5);
   \draw [dashed](6,0) node [anchor=north]{$6$}--(6,1)--(6,5);
   \draw [dashed](7,0) node [anchor=north]{$7$}--(7,1)--(7,5);
   \draw [dashed](8,0) node [anchor=north]{$8$}--(8,1)--(8,5);
   \draw [dashed](9,0) node [anchor=north]{$9$}--(9,1)--(9,5);
   \draw [dashed](10,0) node [anchor=north]{$10$}--(10,1)--(10,5);
   \draw [dashed](0,1) node [anchor=east]{$n$}--(1,1)--(10,1);
   \draw [dashed](0,2) node [anchor=east]{$2n$}--(1,2)--(10,2);
   \draw [dashed](0,3) node [anchor=east]{$3n$}--(1,3)--(10,3);
   \draw [dashed](0,4) node [anchor=east]{$4n$}--(1,4)--(10,4);
   \draw [dashed](0,5) node [anchor=east]{$5n$}--(1,5)--(10,5);
   \draw [dashed](0,5) node [anchor=east]{}--(0,5)--(10,0);
   
\end{tikzpicture}
          \end{tabular}
    \caption{$(2,3n+v_p (a_2))$ is a vertex.}
  \end{center}
\end{figure}

\end{case}

\begin{case}
{\bf \boldmath $(2,3n+v_p (a_2))$ and \boldmath $(5,v_p (a_5))$ are vertices:}
\upshape

In this case, $\mathcal{NP}(f)$ is as in Figure 5. This occurs if and only if $v_p (a_1)\geq0, v_p (a_2) =0, v_p (a_3) \geq n/3, v_p (a_4) \geq 2n/3$ and $v_p (a_5) = n$. Note that $v_p (a_1)\geq 0$ is a trivial condition. Then we can decompose $f(t)$ as $\displaystyle \prod_{i=1}^{10} (t - \alpha_i) \in \mathbb{Q}_p [t]$ so that
\begin{gather*} 
(t- \alpha_1)(t- \alpha_2), (t- \alpha_9)(t- \alpha_{10})\in \mathbb{Q}_p [t], \\
(t - \alpha_3)(t - \alpha_4)(t - \alpha_5), (t - \alpha_6)(t - \alpha_7)(t - \alpha_8)  \in \mathbb{Q}_p [t], \\
v_p (\alpha_1) =v_p (\alpha_2)= n, \ v_p (\alpha_9)=v_p (\alpha_{10}) = 0 ,\\
 v_p (\alpha_3) =v_p (\alpha_4)=v_p (\alpha_5)=2n/3, v_p (\alpha_6)= v_p (\alpha_7)= v_p (\alpha_8) =  n/3.
\end{gather*}
Thus the condition (**) holds if and only if $f(t)$ has no root of valuation $n/3$ or $2n/3$ in $\mathbb{Q}_p$.

\begin{figure}
  \begin{center}
    \begin{tabular}{c}   
         \begin{tikzpicture}
   \draw [thick, -stealth](-0.5,0)--(10.5,0) node [anchor=north]{};
   \draw [thick, -stealth](0,-0.5)--(0,5.5) node{};
   \node [anchor=north west] at (0,0) {O};

   \draw [very thick, domain=0:2, samples=200] plot(\x, -\x+5);
   \draw [very thick, domain=2:5, samples=200] plot(\x, -2/3*\x+13/3);
   \draw [very thick, domain=5:8, samples=200] plot(\x, -1/3*\x+8/3);
   \draw [very thick, domain=8:10, samples=200] plot(\x, 0);

   \draw [dashed](1,0) node [anchor=north]{$1$}--(1,1)--(1,5);
   \draw [dashed](2,0) node [anchor=north]{$2$}--(2,1)--(2,5);
   \draw [dashed](3,0) node [anchor=north]{$3$}--(3,1)--(3,5);
   \draw [dashed](4,0) node [anchor=north]{$4$}--(4,1)--(4,5);
   \draw [dashed](5,0) node [anchor=north]{$5$}--(5,1)--(5,5);
   \draw [dashed](6,0) node [anchor=north]{$6$}--(6,1)--(6,5);
   \draw [dashed](7,0) node [anchor=north]{$7$}--(7,1)--(7,5);
   \draw [dashed](8,0) node [anchor=north]{$8$}--(8,1)--(8,5);
   \draw [dashed](9,0) node [anchor=north]{$9$}--(9,1)--(9,5);
   \draw [dashed](10,0) node [anchor=north]{$10$}--(10,1)--(10,5);
   \draw [dashed](0,1) node [anchor=east]{$n$}--(1,1)--(10,1);
   \draw [dashed](0,2) node [anchor=east]{$2n$}--(1,2)--(10,2);
   \draw [dashed](0,3) node [anchor=east]{$3n$}--(1,3)--(10,3);
   \draw [dashed](0,4) node [anchor=east]{$4n$}--(1,4)--(10,4);
   \draw [dashed](0,5) node [anchor=east]{$5n$}--(1,5)--(10,5);
   \draw [dashed](0,5) node [anchor=east]{}--(0,5)--(10,0);
   
\end{tikzpicture}
          \end{tabular}
    \caption{$(2,3n+v_p (a_2))$ and $(5,v_p (a_5))$ are vertices.}
  \end{center}
\end{figure}

\end{case}
\begin{case}
{\bf \boldmath $(3,2n+v_p (a_3))$ is a vertex:}
\upshape

In this case, there are two possible Newton polygons as in Figure 6. First we consider the lower polygon in Figure 6. This is $\mathcal{NP}(f)$ if and only if $v_p (a_1) \geq 0, v_p (a_2 ) \geq 0, v_p (a_3)=0, v_p (a_4) \geq n/2$ and $v_p (a_5) \geq n$. Note that $v_p (a_1) \geq 0, v_p (a_2 ) \geq 0$ are trivial conditions. Then we can decompose $f(t)$ as $\displaystyle \prod_{i=1}^{10} (t - \alpha_i)$ in $\overline{\mathbb{Q}}_p [t]$ so that
\begin{gather*} 
(t- \alpha_1)(t- \alpha_2)(t-\alpha_3), (t-\alpha_8)(t- \alpha_9)(t- \alpha_{10})\in \mathbb{Q}_p [t], \\
(t - \alpha_4)(t - \alpha_5)(t - \alpha_6)(t - \alpha_7) \in \mathbb{Q}_p [t], \\
v_p (\alpha_1) =v_p (\alpha_2)=v_p (\alpha_3) = n, \  v_p (\alpha_8) =v_p (\alpha_9)=v_p (\alpha_{10}) = 0 ,\\
 v_p (\alpha_4)=v_p (\alpha_5)=v_p (\alpha_6)= v_p (\alpha_7)=  n/2.
\end{gather*}
Thus the condition (**) holds if and only if $f(t)$ has no root of valuation $n/2$ in $\mathbb{Q}_p$.

Second we consider the upper polygon in Figure 6. This is $\mathcal{NP}(f)$ if and only if $v_p (a_1) \geq n/3, v_p (a_2 ) \geq 2n/3, v_p (a_3)=n, v_p (a_4) \geq 3n/2$ and $v_p (a_5) \geq 2n$. Then we can decompose $f(t)$ as $\displaystyle \prod_{i=1}^{10} (t - \alpha_i)$ in $\overline{\mathbb{Q}}_p [t]$ so that
\begin{gather*} 
(t- \alpha_1)(t- \alpha_2)(t-\alpha_3), (t-\alpha_8)(t- \alpha_9)(t- \alpha_{10})\in \mathbb{Q}_p [t], \\
(t - \alpha_4)(t - \alpha_5)(t - \alpha_6)(t - \alpha_7) \in \mathbb{Q}_p [t], \\
v_p (\alpha_1) =v_p (\alpha_2)=v_p (\alpha_3) = 2n/3, \  v_p (\alpha_8) =v_p (\alpha_9)=v_p (\alpha_{10}) = n/3 ,\\
 v_p (\alpha_4)=v_p (\alpha_5)=v_p (\alpha_6)= v_p (\alpha_7)=  n/2.
\end{gather*}
Thus the condition (**) holds if and only if $f(t)$ has no root of valuation $n/3$, $n/2$ or $2n/3$ in $\mathbb{Q}_p$.

\begin{figure}
  \begin{center}
    \begin{tabular}{c}   
         \begin{tikzpicture}
   \draw [thick, -stealth](-0.5,0)--(10.5,0) node [anchor=north]{};
   \draw [thick, -stealth](0,-0.5)--(0,5.5) node{};
   \node [anchor=north west] at (0,0) {O};

   \draw [very thick, domain=0:3, samples=200] plot(\x, -2/3* \x+5);
   \draw [very thick, domain=3:7, samples=200] plot(\x, -\x/2+9/2);
   \draw [very thick, domain=7:10, samples=200] plot(\x, -\x/3 + 10/3);
   
   \draw [very thick, domain=0:3, samples=200] plot(\x, -\x+5);
   \draw [very thick, domain=3:7, samples=200] plot(\x, -\x/2+7/2);
   \draw [very thick, domain=7:10, samples=200] plot(\x, 0);

   \draw [dashed](1,0) node [anchor=north]{$1$}--(1,1)--(1,5);
   \draw [dashed](2,0) node [anchor=north]{$2$}--(2,1)--(2,5);
   \draw [dashed](3,0) node [anchor=north]{$3$}--(3,1)--(3,5);
   \draw [dashed](4,0) node [anchor=north]{$4$}--(4,1)--(4,5);
   \draw [dashed](5,0) node [anchor=north]{$5$}--(5,1)--(5,5);
   \draw [dashed](6,0) node [anchor=north]{$6$}--(6,1)--(6,5);
   \draw [dashed](7,0) node [anchor=north]{$7$}--(7,1)--(7,5);
   \draw [dashed](8,0) node [anchor=north]{$8$}--(8,1)--(8,5);
   \draw [dashed](9,0) node [anchor=north]{$9$}--(9,1)--(9,5);
   \draw [dashed](10,0) node [anchor=north]{$10$}--(10,1)--(10,5);
   \draw [dashed](0,1) node [anchor=east]{$n$}--(1,1)--(10,1);
   \draw [dashed](0,2) node [anchor=east]{$2n$}--(1,2)--(10,2);
   \draw [dashed](0,3) node [anchor=east]{$3n$}--(1,3)--(10,3);
   \draw [dashed](0,4) node [anchor=east]{$4n$}--(1,4)--(10,4);
   \draw [dashed](0,5) node [anchor=east]{$5n$}--(1,5)--(10,5);
   \draw [dashed](0,5) node [anchor=east]{}--(0,5)--(10,0);
   
\end{tikzpicture}
          \end{tabular}
    \caption{$(3,2n+v_p (a_3))$ is a vertex.}
  \end{center}
\end{figure}

\end{case}
\begin{case}
{\bf \boldmath $(4,n+v_p (a_4))$ is a vertex:}
\upshape

In this case, there are two possible Newton polygons as in Figure 7. First we consider the lower polygon in Figure 7. This is $\mathcal{NP}(f)$ if and only if $v_p (a_1) \geq 0, v_p (a_2 ) \geq 0, v_p (a_3)\geq0, v_p (a_4)=0$ and $v_p (a_5) \geq n/2$. Note that $v_p (a_1) \geq 0, v_p (a_2 ) \geq 0, v_p (a_3) \geq 0$ are trivial conditions. Then we can decompose $f(t)$ as $\displaystyle \prod_{i=1}^{10} (t - \alpha_i)$ in $\overline{\mathbb{Q}}_p [t]$ so that
\begin{gather*} 
(t- \alpha_1)\cdots(t-\alpha_4), (t-\alpha_7)\cdots(t- \alpha_{10})\in \mathbb{Q}_p [t], \\
(t - \alpha_5)(t - \alpha_6) \in \mathbb{Q}_p [t], \\
v_p (\alpha_1) =\cdots=v_p (\alpha_4) = n, \  v_p (\alpha_7) =\cdots=v_p (\alpha_{10}) = 0 ,\\
v_p (\alpha_5)=v_p (\alpha_6)= n/2.
\end{gather*}
Thus the condition (**) holds if and only if $f(t)$ has no root of valuation $n/2$ in $\mathbb{Q}_p$.

Second we consider the upper polygon in Figure 7. This is $\mathcal{NP}(f)$ if and only if $v_p (a_1) \geq n/4, v_p (a_2 ) \geq n/2, v_p (a_3)\geq 3n/4, v_p (a_4)=n$ and $v_p (a_5) \geq 3n/2$. Then we can decompose $f(t)$ as $\displaystyle \prod_{i=1}^{10} (t - \alpha_i)$ in $\overline{\mathbb{Q}}_p [t]$ so that
\begin{gather*} 
(t- \alpha_1)\cdots(t-\alpha_4), (t-\alpha_7)\cdots(t- \alpha_{10})\in \mathbb{Q}_p [t], \\
(t - \alpha_5)(t - \alpha_6) \in \mathbb{Q}_p [t], \\
v_p (\alpha_1) =\cdots=v_p (\alpha_4) = 3n/4, \  v_p (\alpha_7) =\cdots=v_p (\alpha_{10}) = n/4 ,\\
v_p (\alpha_5)=v_p (\alpha_6)= n/2.
\end{gather*}
Thus the condition (**) holds if and only if $f(t)$ has no root of valuation $n/4, n/2$ or $3n/4$ and has two irreducible factors of degree $4$ in $\mathbb{Q}_p$.

\begin{figure}
  \begin{center}
    \begin{tabular}{c}   
         \begin{tikzpicture}
   \draw [thick, -stealth](-0.5,0)--(10.5,0) node [anchor=north]{};
   \draw [thick, -stealth](0,-0.5)--(0,5.5) node{};
   \node [anchor=north west] at (0,0) {O};

   \draw [very thick, domain=0:4, samples=200] plot(\x, -3/4* \x+5);
   \draw [very thick, domain=4:6, samples=200] plot(\x, -\x/2+4);
   \draw [very thick, domain=6:10, samples=200] plot(\x, -\x/4 + 5/2);
   
   \draw [very thick, domain=0:4, samples=200] plot(\x, -\x+5);
   \draw [very thick, domain=4:6, samples=200] plot(\x, -\x/2+3);
   \draw [very thick, domain=6:10, samples=200] plot(\x, 0);

   \draw [dashed](1,0) node [anchor=north]{$1$}--(1,1)--(1,5);
   \draw [dashed](2,0) node [anchor=north]{$2$}--(2,1)--(2,5);
   \draw [dashed](3,0) node [anchor=north]{$3$}--(3,1)--(3,5);
   \draw [dashed](4,0) node [anchor=north]{$4$}--(4,1)--(4,5);
   \draw [dashed](5,0) node [anchor=north]{$5$}--(5,1)--(5,5);
   \draw [dashed](6,0) node [anchor=north]{$6$}--(6,1)--(6,5);
   \draw [dashed](7,0) node [anchor=north]{$7$}--(7,1)--(7,5);
   \draw [dashed](8,0) node [anchor=north]{$8$}--(8,1)--(8,5);
   \draw [dashed](9,0) node [anchor=north]{$9$}--(9,1)--(9,5);
   \draw [dashed](10,0) node [anchor=north]{$10$}--(10,1)--(10,5);
   \draw [dashed](0,1) node [anchor=east]{$n$}--(1,1)--(10,1);
   \draw [dashed](0,2) node [anchor=east]{$2n$}--(1,2)--(10,2);
   \draw [dashed](0,3) node [anchor=east]{$3n$}--(1,3)--(10,3);
   \draw [dashed](0,4) node [anchor=east]{$4n$}--(1,4)--(10,4);
   \draw [dashed](0,5) node [anchor=east]{$5n$}--(1,5)--(10,5);
   \draw [dashed](0,5) node [anchor=east]{}--(0,5)--(10,0);
   
\end{tikzpicture}
          \end{tabular}
    \caption{$(4,n+v_p (a_4))$ is a vertex.}
  \end{center}
\end{figure}

\end{case}
\begin{case}
{\bf \boldmath $(5, v_p (a_5))$ is a vertex:}
\upshape

In this case, there are three possible Newton polygons as in Figure 8. First we consider the bottom polygon in Figure 8. This is $\mathcal{NP}(f)$ if and only if $v_p (a_1) \geq 0, v_p (a_2 ) \geq 0, v_p (a_3)\geq0, v_p (a_4)\geq0$ and $v_p (a_5) =0$. Note that $v_p (a_1) \geq 0, v_p (a_2 ) \geq 0, v_p (a_3) \geq 0, v_p (a_4) \geq 0$ are trivial conditions. Then we can decompose $f(t)$ as $\displaystyle \prod_{i=1}^{10} (t - \alpha_i)$ in $\overline{\mathbb{Q}}_p [t]$ so that
\begin{gather*} 
(t- \alpha_1)\cdots(t-\alpha_5), (t-\alpha_6)\cdots(t- \alpha_{10})\in \mathbb{Q}_p [t], \\
v_p (\alpha_1) =\cdots=v_p (\alpha_5) = n, \  v_p (\alpha_7) =\cdots=v_p (\alpha_{10}) = 0 ,
\end{gather*}
The condition (**) always holds.

Second we consider the middle polygon in Figure 8. This is $\mathcal{NP}(f)$ if and only if $v_p (a_1) \geq n/5, v_p (a_2 ) \geq 2n/5, v_p (a_3)\geq 3n/5, v_p (a_4)\geq 4n/5$ and $v_p (a_5) =n$. Then we can decompose $f(t)$ as $\displaystyle \prod_{i=1}^{10} (t - \alpha_i)$ in $\overline{\mathbb{Q}}_p [t]$ so that
\begin{gather*} 
(t- \alpha_1)\cdots(t-\alpha_5), (t-\alpha_6)\cdots(t- \alpha_{10})\in \mathbb{Q}_p [t], \\
v_p (\alpha_1) =\cdots=v_p (\alpha_5) = 4n/5, \  v_p (\alpha_6) =\cdots=v_p (\alpha_{10}) = n/5.
\end{gather*}
Thus the condition (**) holds if and only if $f(t)$ has two irreducible factors of degree $5$ in $\mathbb{Q}_p$.

Third we consider the top polygon in Figure 8. This is $\mathcal{NP}(f)$ if and only if $v_p (a_1) \geq 2n/5, v_p (a_2 ) \geq 4n/5, v_p (a_3)\geq 6n/5, v_p (a_4)\geq 8n/5$ and $v_p (a_5) =2n$. Then we can decompose $f(t)$ as $\displaystyle \prod_{i=1}^{10} (t - \alpha_i)$ in $\overline{\mathbb{Q}}_p [t]$ so that
\begin{gather*} 
(t- \alpha_1)\cdots(t-\alpha_5), (t-\alpha_6)\cdots(t- \alpha_{10})\in \mathbb{Q}_p [t], \\
v_p (\alpha_1) =\cdots=v_p (\alpha_5) = 3n/5, \  v_p (\alpha_6) =\cdots=v_p (\alpha_{10}) = 2n/5.
\end{gather*}
Thus the condition (**) holds if and only if $f(t)$ has two irreducible factors of degree $5$ in $\mathbb{Q}_p$.

\begin{figure}
  \begin{center}
    \begin{tabular}{c}   
         \begin{tikzpicture}
   \draw [thick, -stealth](-0.5,0)--(10.5,0) node [anchor=north]{};
   \draw [thick, -stealth](0,-0.5)--(0,5.5) node{};
   \node [anchor=north west] at (0,0) {O};

   \draw [very thick, domain=0:5, samples=200] plot(\x, -\x+5);
   \draw [very thick, domain=5:10, samples=200] plot(\x, 0);
   
   \draw [very thick, domain=0:5, samples=200] plot(\x, -4/5*\x+5);
   \draw [very thick, domain=5:10, samples=200] plot(\x, -\x/5 + 2);
   
   \draw [very thick, domain=0:5, samples=200] plot(\x, -3/5*\x+5);
   \draw [very thick, domain=5:10, samples=200] plot(\x, -2/5*\x + 4);

   \draw [dashed](1,0) node [anchor=north]{$1$}--(1,1)--(1,5);
   \draw [dashed](2,0) node [anchor=north]{$2$}--(2,1)--(2,5);
   \draw [dashed](3,0) node [anchor=north]{$3$}--(3,1)--(3,5);
   \draw [dashed](4,0) node [anchor=north]{$4$}--(4,1)--(4,5);
   \draw [dashed](5,0) node [anchor=north]{$5$}--(5,1)--(5,5);
   \draw [dashed](6,0) node [anchor=north]{$6$}--(6,1)--(6,5);
   \draw [dashed](7,0) node [anchor=north]{$7$}--(7,1)--(7,5);
   \draw [dashed](8,0) node [anchor=north]{$8$}--(8,1)--(8,5);
   \draw [dashed](9,0) node [anchor=north]{$9$}--(9,1)--(9,5);
   \draw [dashed](10,0) node [anchor=north]{$10$}--(10,1)--(10,5);
   \draw [dashed](0,1) node [anchor=east]{$n$}--(1,1)--(10,1);
   \draw [dashed](0,2) node [anchor=east]{$2n$}--(1,2)--(10,2);
   \draw [dashed](0,3) node [anchor=east]{$3n$}--(1,3)--(10,3);
   \draw [dashed](0,4) node [anchor=east]{$4n$}--(1,4)--(10,4);
   \draw [dashed](0,5) node [anchor=east]{$5n$}--(1,5)--(10,5);
   \draw [dashed](0,5) node [anchor=east]{}--(0,5)--(10,0);
   
\end{tikzpicture}
          \end{tabular}
    \caption{$(5, v_p (a_5))$ is a vertex.}
  \end{center}
\end{figure}

\end{case}
\begin{case}
{\bf \boldmath $\mathcal{NP}(f)$ is a line:}
\upshape

In this case, $\mathcal{NP}(f)$ is as in Figure 9. This occurs if and only if $v_p (a_1)\geq n/2, v_p (a_2) \geq n, v_p (a_3) \geq 3n/2, v_p (a_4) \geq 2n$ and $v_p (a_5) \geq 5n/2$. Then we can decompose $f(t)$ as $\displaystyle \prod_{i=1}^{10} (t - \alpha_i)$ in $\overline{\mathbb{Q}}_p [t]$ so that
\begin{gather*} 
(t- \alpha_1)\cdots(t- \alpha_{10})\in \mathbb{Q}_p [t], \\
v_p (\alpha_1) =\cdots = v_p (\alpha_{10}) = n/2.
\end{gather*}
Thus the condition (**) holds if and only if $f(t)$ has no root of valuation $n/2$ nor a factor of degree $3$ or $5$ in $\mathbb{Q}_p$.

\begin{figure}
  \begin{center}
    \begin{tabular}{c}   
         \begin{tikzpicture}
   \draw [thick, -stealth](-0.5,0)--(10.5,0) node [anchor=north]{};
   \draw [thick, -stealth](0,-0.5)--(0,5.5) node{};
   \node [anchor=north west] at (0,0) {O};

   \draw [very thick, domain=0:10, samples=200] plot(\x, -1/2*\x+5);

   \draw [dashed](1,0) node [anchor=north]{$1$}--(1,1)--(1,5);
   \draw [dashed](2,0) node [anchor=north]{$2$}--(2,1)--(2,5);
   \draw [dashed](3,0) node [anchor=north]{$3$}--(3,1)--(3,5);
   \draw [dashed](4,0) node [anchor=north]{$4$}--(4,1)--(4,5);
   \draw [dashed](5,0) node [anchor=north]{$5$}--(5,1)--(5,5);
   \draw [dashed](6,0) node [anchor=north]{$6$}--(6,1)--(6,5);
   \draw [dashed](7,0) node [anchor=north]{$7$}--(7,1)--(7,5);
   \draw [dashed](8,0) node [anchor=north]{$8$}--(8,1)--(8,5);
   \draw [dashed](9,0) node [anchor=north]{$9$}--(9,1)--(9,5);
   \draw [dashed](10,0) node [anchor=north]{$10$}--(10,1)--(10,5);
   \draw [dashed](0,1) node [anchor=east]{$n$}--(1,1)--(10,1);
   \draw [dashed](0,2) node [anchor=east]{$2n$}--(1,2)--(10,2);
   \draw [dashed](0,3) node [anchor=east]{$3n$}--(1,3)--(10,3);
   \draw [dashed](0,4) node [anchor=east]{$4n$}--(1,4)--(10,4);
   \draw [dashed](0,5) node [anchor=east]{$5n$}--(1,5)--(10,5);

\end{tikzpicture}
          \end{tabular}
    \caption{$\mathcal{NP}(f)$ is a line.}
  \end{center}
\end{figure}

\end{case}
This completes the proof of Theorem \ref{theo1.3}.
\end{proof}

\section*{Acknowledgement}
I would like to appreciate Prof. Akio Tamagawa whose comments and suggestions were of inestimable value for my study. I also would like to appreciate Prof. Akihiko Yukie for his many valuable opinions and advices on the presentation of this paper.

 
 
\end{document}